\documentclass{amsart}
\usepackage{amssymb}
\usepackage{amsmath}
\usepackage{amsfonts}

\theoremstyle{plain}

\newtheorem{corollary}{Corollary}

\newtheorem{definition}{Definition}

\newtheorem{lemma}{Lemma}

\newtheorem{proposition}{Proposition}
\newtheorem{remark}{Remark}

\numberwithin{equation}{section}
\input{tcilatex}

\begin{document}
\title[Mollifiers in Clifford Analysis]{Mollifiers in Clifford Analysis}
\author{Dejenie A. Lakew }
\address{Virginia State University\\
Department of Mathematics \& Computer Science}
\email{dlakew@vsu.edu}
\urladdr{http://www.vsu.edu}
\date{January 07, 2008}
\dedicatory{Dedicated to the Memory of my Dear Mother}
\thanks{This paper is in final form and no version of it will be submitted
for publication elsewhere.}

\begin{abstract}
We introduce mollifiers in Clifford analysis setting and construct a
sequence of $C^{\infty }-$functions that approximates a $\gamma -$regular
function and also a solution to a non homogeneous BVP of an in-homogeneous
Dirac like operator in certain Sobolev spaces over bounded domains whose
boundary is not that wild. One can extend the smooth functions upto the
boundary if the domain has a $\ C^{1}-$ boundary and this is the case in the
paper as we consider a domain whose boundary is a $C^{2}-$hyper surface.
\end{abstract}

\maketitle

\section{\textbf{Introduction: Algebraic and Analytic Rudiments}}

Let $\Omega $ be a bounded domain in $%
%TCIMACRO{\U{211d} }%
%BeginExpansion
\mathbb{R}
%EndExpansion
^{n}$ whose boundary is a $C^{2}-$hyper surface and $Cl_{n}$ be a $2^{n}$-
dimensional Clifford algebra generated by $%
%TCIMACRO{\U{211d} }%
%BeginExpansion
\mathbb{R}
%EndExpansion
^{n}$ with an inner product that satisfies $x^{2}=-\parallel x\parallel ^{2}$%
.

Then for $e_{1},e_{2},...,e_{n}$ which are orthonormal basis of $%
%TCIMACRO{\U{211d} }%
%BeginExpansion
\mathbb{R}
%EndExpansion
^{n}$ , we have an equality $e_{ij}+e_{ji}=-\delta _{ij}e_{0}$ , with $%
\delta _{ij}$ , the Kronecker delta symbol and $e_{0}$, the identity element
of the Clifford algebra.

\ \ \ \ 

A $Cl_{n}-$valued function $\ f\ $\ defined in $\Omega $ has a standard
representation :

\begin{equation}
\ f(x)=\sum_{A}e_{A}f_{A}(x),x\in \Omega  \label{formula1}
\end{equation}

where for each index set $A,$ $f_{A}:\Omega \rightarrow 
%TCIMACRO{\U{211d} }%
%BeginExpansion
\mathbb{R}
%EndExpansion
$ is a real valued section of $f$.

\ \ 

Such a function $f$ \ is continuous, differentiable, integrable, measurable,
etc. over $\Omega $, if each section $f_{A}$ \ of $f$\ \ is respectively
continuous, differentiable, integrable, measurable, etc. over $\Omega $.

\bigskip\ \ \ 

Thus the usual function spaces, the H\"{o}lder spaces denoted by $C^{\alpha
}(\Omega ,Cl_{n}),C^{m,\alpha }(\Omega ,Cl_{n})$ and the Sobolev spaces
denoted by $W^{p,k}(\Omega ,Cl_{n})$ for $m,k=0,1,...$ and $1<p<\infty $,
are defined as follows:

\ 

$\ f\in C^{\alpha }(\Omega ,Cl_{n})$ iff $f_{A}\in C^{\alpha }(\Omega ,%
%TCIMACRO{\U{211d} }%
%BeginExpansion
\mathbb{R}
%EndExpansion
)$ where $C^{\alpha }(\Omega ,%
%TCIMACRO{\U{211d} }%
%BeginExpansion
\mathbb{R}
%EndExpansion
)$ is the space of all functions $f$\ which are H\"{o}lder continuous with H%
\"{o}lder exponent $\alpha $ :

\begin{equation*}
\mid f\left( x\right) -f\left( y\right) \mid \leq k_{f}\mid x-y\mid ^{\alpha
}
\end{equation*}%
for $x,y\in \Omega $ with norm given by : 
\begin{equation*}
\Vert f\Vert _{C^{\alpha }\left( \Omega ,%
%TCIMACRO{\U{211d} }%
%BeginExpansion
\mathbb{R}
%EndExpansion
\right) }=\parallel f\parallel _{C\left( \Omega ,%
%TCIMACRO{\U{211d} }%
%BeginExpansion
\mathbb{R}
%EndExpansion
\right) }+\underset{\underset{}{\underset{x\neq y}{x,y\in \Omega }}}{\sup }%
\frac{\mid f\left( x\right) -f\left( y\right) \mid }{\mid x-y\mid ^{\alpha }}
\end{equation*}%
where $k_{f}$ is a positive constant which is specific to the particular
function $f$ .

\ 

For a very trivial fact, the H\"{o}lder exponent $\alpha $ should be in the
interval $(0,1]$, for otherwise, if $\alpha >1,$ we have 
\begin{equation*}
\frac{\mid f\left( x\right) -f\left( y\right) \mid }{\mid x-y\mid }\leq
k_{f}\mid x-y\mid ^{\zeta }
\end{equation*}%
for some $\zeta >0$ and some positive constant $k_{f}$ . Then one can see
that $f$ has a zero derivative at every point $x$ of the domain and
therefore it is a cons$\tan $t . That is, functions whose H\"{o}lder
exponents are greater than one are constant.

\ \ \ 

When $\alpha =1,$ the functions are called Lipschitz functions and these
functions have bounded derivatives over the domain $\Omega $.

\ \ \ 

Also, $\ f\in C^{m,\alpha }(\Omega ,Cl_{n})$ iff $f_{A}\in C^{m,\alpha
}(\Omega ,%
%TCIMACRO{\U{211d} }%
%BeginExpansion
\mathbb{R}
%EndExpansion
)$ where $C^{m,\alpha }(\Omega ,%
%TCIMACRO{\U{211d} }%
%BeginExpansion
\mathbb{R}
%EndExpansion
)$ is the space of functions $f:\Omega \rightarrow 
%TCIMACRO{\U{211d} }%
%BeginExpansion
\mathbb{R}
%EndExpansion
$\ which are $m-$times continuously differentiable and whose $m-$th
derivative is H\"{o}lder continuous with exponent $\alpha $ and with norm
given by 
\begin{eqnarray*}
&\parallel &f\parallel _{C^{m,\alpha }\left( \Omega ,%
%TCIMACRO{\U{211d} }%
%BeginExpansion
\mathbb{R}
%EndExpansion
\right) }=\parallel f\parallel _{C^{m-1}\left( \Omega ,%
%TCIMACRO{\U{211d} }%
%BeginExpansion
\mathbb{R}
%EndExpansion
\right) }+\parallel f^{\left( m\right) }\parallel _{C^{\alpha }\left( \Omega
,%
%TCIMACRO{\U{211d} }%
%BeginExpansion
\mathbb{R}
%EndExpansion
\right) } \\
&=&\parallel f\parallel _{C^{m}\left( \Omega ,%
%TCIMACRO{\U{211d} }%
%BeginExpansion
\mathbb{R}
%EndExpansion
\right) }+\underset{\underset{x\neq y}{x,y\in \Omega }}{\sup }\frac{\mid
f^{\left( m\right) }\left( x\right) -f^{\left( m\right) }\left( y\right)
\mid }{\mid x-y\mid ^{\alpha }}
\end{eqnarray*}

Finally for $p\in \lbrack 1,\infty ),$ Sobolev spaces are defined in a
similar way:

$f\in W^{p,k}(\Omega ,Cl_{n})$ iff $f_{A}\in W^{p,k}(\Omega ,%
%TCIMACRO{\U{211d} }%
%BeginExpansion
\mathbb{R}
%EndExpansion
)$ where $W^{p,k}(\Omega ,%
%TCIMACRO{\U{211d} }%
%BeginExpansion
\mathbb{R}
%EndExpansion
)$ is the space of real valued functions $f$ defined over $\Omega $ which
are locally $p-$ integrable over $\Omega $ and whose $j-th$ distributional (
or weak) derivatives $D^{j}f$\ \ with $\mid j\mid \leq k$ exist and are all $%
p-$integrable over $\Omega $ and norm in such a space is defined as : 
\begin{equation*}
\parallel f\parallel _{W^{p,k}\left( \Omega ,%
%TCIMACRO{\U{211d} }%
%BeginExpansion
\mathbb{R}
%EndExpansion
\right) }=\left( \dsum\limits_{\mid j\mid \leq k}\parallel D^{j}f\parallel
_{L^{p}\left( \Omega ,%
%TCIMACRO{\U{211d} }%
%BeginExpansion
\mathbb{R}
%EndExpansion
\right) }^{p}\right) ^{p^{-1}}
\end{equation*}

Here, a locally integrable function $f:\Omega \rightarrow 
%TCIMACRO{\U{211d} }%
%BeginExpansion
\mathbb{R}
%EndExpansion
$\ is said to have a locally integrable $j-th$ order distributional (or
weak) derivative over $\Omega $ if and only if

\begin{equation*}
\dint\limits_{\Omega }D^{j}f\left( x\right) \phi \left( x\right) d\Omega
_{x}=\left( -1\right) ^{\mid j\mid }\dint\limits_{\Omega }f\left( x\right)
D^{j}\phi \left( x\right)
\end{equation*}

for all test functions $\phi \in C_{c}^{\infty }\left( \Omega \right) $, and
\ \ \ $D^{j}=\dprod\limits_{i=1}^{n}\frac{\partial ^{j_{i}}}{\partial
x_{i}^{j_{i}}}$ \ with $j$ a multi-index exponent such that $%
\dsum\limits_{i=1}^{n}j_{i}=j$.

Note here that $W^{p,0}(\Omega ,Cl_{n})=L^{p}(\Omega ,Cl_{n}),$ the Lebesgue
space of $p-$integrable Clifford valued functions and \ for a detail study
of function spaces, one can refer\ \cite{dr1,dr2,el,sm}

\ \ \ \ \ \ 

For $p=2$, the Lebesgue space $L^{2}(\Omega ,Cl_{n})$ becomes a Hilbert
space with a Clifford-valued inner product given by

\begin{equation}
\langle f,g\rangle _{\Omega }:=\int_{\Omega }\overline{f(x)}g(x)d\Omega
\label{innerproduct}
\end{equation}

\ \ \ \ \ \ 

\ \ \ 

Introduce the in-homogeneous \textbf{Dirac}-operator with gradient potential 
$\gamma $ by:

\begin{equation}
D_{\gamma }:=\sum_{j=1}^{n}e_{j}\left( \frac{\partial }{\partial x_{j}}%
-\gamma _{j}\right)  \label{inhomdirac}
\end{equation}

\bigskip where $\gamma =\sum_{j=1}^{n}e_{j}\frac{\partial }{\partial x_{j}}%
\Gamma $ ( with $\Gamma \in C^{1}(\Omega \rightarrow 
%TCIMACRO{\U{211d} }%
%BeginExpansion
\mathbb{R}
%EndExpansion
)$ linear) is called the \textit{gradient potential} of $\ \Gamma $

\begin{definition}
A function $f\in C^{1}(\Omega \rightarrow Cl_{n})$ is said to be left $%
\gamma $\textbf{-regular }if $D_{\gamma }f(x)=0$, $\forall x\in \Omega $ and
right $\gamma $\textbf{-regular} \ if $f(x)D_{\gamma }=0$.
\end{definition}

An example of a function which is both left and right $\gamma -$regular over 
$\Omega $ is given by

\begin{equation}
\Psi ^{\Gamma }(x):=\frac{\overline{x}}{\omega _{n}\Vert x\Vert ^{n}}%
e^{-\Gamma (x)}  \label{fundsol}
\end{equation}

where $\omega _{n}=\frac{\sqrt{\pi ^{n}}}{\Gamma \left( \frac{n+2}{2}\right) 
}$ is the surface area of the unit sphere in $%
%TCIMACRO{\U{211d} }%
%BeginExpansion
\mathbb{R}
%EndExpansion
^{n}$.

The function given above is also called a \textbf{fundamental solution} (or 
\textbf{Cauchy kernel}) for the in-homogeneous Dirac operator $D_{\gamma }$

\begin{proposition}
Let $\Omega $ be a bounded, $C^{2}-$domain in $%
%TCIMACRO{\U{211d} }%
%BeginExpansion
\mathbb{R}
%EndExpansion
^{n}$ and let $g\in W_{\Gamma }^{2,k-\frac{1}{2}}\left( \partial \Omega
,Cl_{n}\right) $, $k=1,2,...$. Then the
\end{proposition}

$\ $%
\begin{equation}
\ \mathbf{BVP}:\ \ \ \ \ \left\{ 
\begin{array}{cc}
D_{\gamma }f=0 & \text{on }\Omega \\ 
trf=g & \text{on }\partial \Omega%
\end{array}%
\right.  \label{bvp}
\end{equation}

has a solution $f\in W_{\Gamma }^{2,k}\left( \Omega ,Cl_{n}\right) $ given by

\begin{equation}
f(x)=\int_{\partial \Omega }\Psi _{\Gamma }(x-y)\nu (y)g(y)d\Sigma _{y},\
x\in \Omega  \label{solu}
\end{equation}

\bigskip\ \ \ \ \ 

The theme here is to construct Clifford valued $C^{\infty }-$ function $g$ \
over $\Omega $ \ that approximates the solution function $f$ \ in the $%
H^{k}\left( \Omega ,Cl_{n}\right) $ $\left( \text{or }W^{2,k}\left( \Omega
,Cl_{n}\right) \right) $ sense and also to approximate the solution of a non
homogeneous boundary value problem on Sobolev spaces based at $L^{p}\left(
\Omega ,Cl_{n}\right) $ : 
\begin{equation*}
\mathbf{NHBVP}:\left\{ 
\begin{array}{l}
\begin{array}{cc}
D_{\gamma }f=h & \text{on }\Omega%
\end{array}
\\ 
\begin{array}{cc}
trf=g & \text{on }\partial \Omega%
\end{array}%
\end{array}%
\right.
\end{equation*}%
whose solution is given by : $W^{p,k}\left( \Omega ,Cl_{n}\right) \ni f=\Psi
^{\Gamma }\ast _{\mid \partial \Omega }\nu \left( tr_{\partial \Omega
}f\right) +\Psi ^{\Gamma }\ast _{\mid \Omega }\left( D_{\gamma }f\right) $
and substituting $tr_{\Omega }f=g$ and $D_{\gamma }f=h$ on $\Omega $, where $%
g\in W^{p,k-\frac{1}{p}}\left( \partial \Omega ,Cl_{n}\right) $ and $h\in
W^{p,k-1}\left( \Omega ,Cl_{n}\right) $ where the result is given in \textbf{%
Proposition} 6.

\ \ \ 

This is possible by constructing a smooth function $g$\ over any sub domain $%
\Delta \subset \subset \Omega ,$ where for each $\delta >0$, we have that 
\begin{equation*}
\Vert f-g\Vert _{W_{\Gamma }^{2,k}\left( \Delta ,Cl_{n}\right) }<\delta
\end{equation*}
and taking the supremum of such approximations over all such possible sub
domains as

\begin{equation*}
\underset{\Delta \subset \subset \Omega }{\sup }\Vert f-g\Vert _{W_{\Gamma
}^{2,k}\left( \Delta ,Cl_{n}\right) }
\end{equation*}%
we get the result.

\ \ \ \ 

The smooth functions in general are constructed using mollifiers which sooth
locally or globally integrable functions in certain Sobolev spaces and the
notation $\Delta \subset \subset \Omega $ read as "$\Delta $ is compactly
contained in $\Omega $ " is to mean that $\Delta $ is a subset of $\Omega $
whose compact closure $\overline{\Delta }$ is also contained in $\Omega $.

\ \ \ 

In \cite{d1} , the author constructed a family of functions which are called
minimal to approximate in the best way, such a $\gamma -$regular function
with finitely many of these functions. \ \ For detail results, see the
reference therein.

\ \ \ 

\ \ 

\section{\textbf{Approximations with Smooth Functions}}

As I mentioned above, in \cite{d1} we construct $Cl_{n}$-minimal family of
functions in $B_{\gamma }^{2}(\Omega ,Cl_{n})$ which are used for
approximating solutions of elliptic boundary value problems in the best way.
The construction was made by choosing dense points of some outer surface and
define a family of functions from the fundamental solution $\Psi ^{\Gamma }$
of the in-homogeneous Dirac operator $D_{\gamma }$ with the selected points
as the singular points of the fundamental solution. We then refine these
functions more by an orthogonalization like process. The approximating
functions constructed in this way were in the Sobolev space where the
function to be approximated belongs.

\bigskip\ \ \ \ \ \ \ \ \ \ 

But what we intend to do here is that the same function which is
approximated by minimal family of functions can also be approximated by
smooth functions( in fact $C^{\infty }-$functions ) over the domain $\Omega $%
.

\ \ \ 

We shall mention that the smooth approximation over the domain is always
possible as long as the function is integrable over the domain, and this
approximation is extendable up to the boundary if the boundary of the domain
is a $C^{1}-$ hypersurface. Therefore, when the domain is Lipschitz( minimal
smoothness condition on the boundary), the approximating smooth functions
may not be extendable up to the boundary.

\bigskip\ \ \ 

We therefore start with the notion of a mollifier. As a $Cl_{n}-$valued
function $f$ has a general representation given by $\left( \text{\ref%
{formula1}}\right) $, we start with mollifying a real valued function and
then we extend that definition to that of a Clifford valued function.

\bigskip\ \ \ 

Let $\Omega $ be a bounded domain with a $C^{1}-$boundary, and for $\epsilon 
$ be a positive constant, define a sub domain $\Omega _{\epsilon }$ of $%
\Omega $ by $\Omega _{\epsilon }:=\left\{ x\in \Omega :\text{dist}\left(
x,\partial \Omega \right) >\epsilon \right\} $.

Let us also consider the function 
\begin{equation}
\phi \left( x\right) =\chi _{\overset{0}{B\left( 0,1\right) }}ke^{\left(
\parallel x\parallel ^{2}-1\right) ^{-1}}
\end{equation}

which is a $C^{\infty }-$function over $%
%TCIMACRO{\U{211d} }%
%BeginExpansion
\mathbb{R}
%EndExpansion
^{n}$ whose compact support is within the unit ball $B\left( 0,1\right) $
and we choose the constant $k$ so that the integral of $\phi $ over the
space $%
%TCIMACRO{\U{211d} }%
%BeginExpansion
\mathbb{R}
%EndExpansion
^{n}$ is a unit. The function $\chi _{\overset{0}{B}}$ is the characteristic
function of the interior of the unit ball $B\left( 0,1\right) $.

Then for a function $f:\Omega \rightarrow 
%TCIMACRO{\U{211d} }%
%BeginExpansion
\mathbb{R}
%EndExpansion
$ which is locally integrable, we define the convolution :%
\begin{equation}
f^{\epsilon }\left( x\right) :=\dint\limits_{\Omega }\epsilon ^{-n}\phi
\left( \frac{x-y}{\epsilon }\right) f\left( y\right) d\Omega _{y}
\end{equation}%
which is the convolution of the mollifier function $\phi _{\epsilon }$ with
that of $f$ over the sub domain $\Omega _{\epsilon }$ where $\phi _{\epsilon
}\left( x\right) =\epsilon ^{-n}\phi \left( \epsilon ^{-1}x\right) $ is a $%
C^{\infty }-$function compactly supported in the $\epsilon -$ball centered
at the origin. The above function $f^{\epsilon }$ defined as $f^{\epsilon
}:=\phi _{\epsilon }\ast f$ \ is some times called \ a regularization of $f$.

\begin{lemma}
The convolution function $f^{\epsilon }$ is a $C^{\infty }-$function over
the $\epsilon -$thick skin removed sub domain $\Omega _{\epsilon }$ and
besides $\underset{\epsilon \downarrow 0}{\lim }f^{\epsilon }=f$ in measure.
\end{lemma}

\begin{proposition}
(Clifford Analysis version of a regularization)

For $f=\sum_{A}e_{A}f_{A}:\Omega \rightarrow Cl_{n}$ and $f_{A}^{\epsilon
}:=\phi _{\epsilon }\ast f_{A}$, the regularization $f^{\epsilon
}=\dsum\limits_{A}e_{A}f_{A}^{\epsilon }$ is $C^{\infty }-$over $\Omega
_{\epsilon }$ and further more $\underset{\epsilon \rightarrow 0}{\lim }%
\left( \dsum\limits_{A}e_{A}f_{A}^{\epsilon }\right) =f$
\end{proposition}

\begin{proof}
For each index set $A,$ $f_{A}$ is a real valued function from the domain $%
\Omega $ and by the above lemma, the convolution $f_{A}^{\epsilon }=\phi
_{\epsilon }\ast f_{A}$ is a $C^{\infty }-$function over $\Omega _{\epsilon
} $. Then the Clifford sum of such smooth functions : $\dsum\limits_{A}e_{A}%
\left( \phi _{\epsilon }\ast f_{A}\right) =:f^{\epsilon }$ is a smooth
function as well. Also by continuity, $\dsum\limits_{A}e_{A}\left( \phi
_{\epsilon }\ast f_{A}\right) \rightarrow \dsum\limits_{A}e_{A}f_{A}$ as $%
\epsilon \rightarrow 0$, that is $f^{\epsilon }\rightarrow f$ as $\epsilon
\rightarrow 0$.
\end{proof}

The following proposition is the main result of the paper.

\begin{proposition}
Let $\Omega $ be a bounded domain with a $C^{1}-$ boundary and $\ 1<p<\infty
,$ $f\in W^{p,k}(\Omega ,Cl_{n})$ where, $k=0,1,2,...$. Then $\forall
\epsilon >0,$ there exists a $Cl_{n}-$ valued function $\Psi
=\dsum\limits_{A}e_{A}\psi _{A}$ over $\Omega $ which is $C^{\infty }-$ up
to the boundary such that 
\begin{equation*}
\parallel f-\Psi \parallel _{W^{p,k}\left( \Omega \cup \partial \Omega
,Cl_{n}\right) }<\epsilon
\end{equation*}
\end{proposition}

\begin{proof}
We first start with the Clifford Analysis version of regularization .

For a $Cl_{n}-$valued function $f$ defined on $\Omega $ which is represented
by

$\ f(x)=\sum_{A}e_{A}f_{A}(x)$, we construct locally integrable $C^{\infty
}- $functions from $f$ as%
\begin{equation}
f^{\epsilon }:=\sum_{A}e_{A}f_{A}^{\epsilon }  \label{clifmolli}
\end{equation}%
\ where, for each $A$, 
\begin{equation*}
f_{A}^{\epsilon }=\phi _{\epsilon }\ast f_{A}
\end{equation*}%
From the construction of the mollifiers $\phi _{\epsilon }$, one can show
that the $\epsilon -$wide section $f^{\epsilon }$ of the Clifford valued
function $f$ is $C^{\infty }-$ function over the sub domain $\Omega
_{\epsilon }$ as each component function $f_{A}^{\epsilon }$ is $C^{\infty
}- $ over $\Omega _{\epsilon }$ and 
\begin{equation*}
\underset{\epsilon \downarrow 0}{\lim }\left( \sum_{A}e_{A}\left( f_{A}\ast
\phi _{\epsilon }\right) \right) =f=\sum_{A}e_{A}f_{A}
\end{equation*}%
in measure over $\Omega $.

\ 

The next procedure is to look at how each $%
%TCIMACRO{\U{211d} }%
%BeginExpansion
\mathbb{R}
%EndExpansion
-$valued component function $f_{A}$\ of the $Cl_{n}-$valued function $f$ is
approximated by $C^{\infty }-$functions (for more information on this
particular procedure, one can refer \cite{el}).

\ 

The process is out lined next, where some kind of \textit{surgery} on the
domain $\Omega $ is performed in order to construct smooth functions that
will approximate $f_{A}$ in terms of other smooth functions called
partitions of unity (refer \cite{el} for details) and then we extend the
result to work for a $Cl_{n}-$valued function $f$ .

\ \ 

To explain exactly what is happening is that we cut off each component
function which is in a Sobolev space that may have a singularity of some
order, by $C^{\infty }-$ functions which control the singularity and sooth
the function and then we patch the smooth sections to create the needed $%
C^{\infty }-$\ approximating functions.

\ \ 

Thus, for each $i,\left( i=1,2,..,\right) $ construct sub domain $\Omega
_{i}:=\left\{ x\in \Omega :\text{dist}\left( x,\partial \Omega \right)
>i^{-1}\right\} $ so that $\Omega =\dbigcup\limits_{i=1}^{\infty }\Omega
_{i} $ and also consider the decomposition of the domain in the following
way: $\widetilde{\Omega }_{i}=\Omega _{i+3}-\overline{\Omega }_{i+1}$, and
then pick a sub domain $\widetilde{\Omega }_{0}\subset \subset \Omega $ so
that $\Omega =\dbigcup\limits_{i=0}^{\infty }\widetilde{\Omega }_{i}$. Then
for an $%
%TCIMACRO{\U{211d} }%
%BeginExpansion
\mathbb{R}
%EndExpansion
-$valued component function $f_{A}\in W^{p,k}\left( \Omega \right) $ of the
Clifford valued function $f$ and for any partition of unity $\left\{ \theta
_{i}\right\} _{i=0}^{\infty }$ associated to the open cover $\left\{ 
\widetilde{\Omega }_{i}\right\} _{i=0}^{\infty }$ of $\Omega $, the function 
$\theta _{i}f_{A}$ is compactly supported over $\widetilde{\Omega }_{i}$ and
furthermore, it is \ in the Sobolev space $W^{p,k}\left( \Omega ,%
%TCIMACRO{\U{211d} }%
%BeginExpansion
\mathbb{R}
%EndExpansion
\right) $.

\ \ \ \ 

Let us consider a $\beta >0$ and choose a positive but small number $%
\epsilon _{i}$ such that the convolution function $\phi _{\in _{i}}\ast
\left( \theta _{i}f_{A}\right) =:g_{i}$ has a compact support in $%
V_{i}:=\Omega _{i+4}-\overline{\Omega }_{i}$ \ which contains $\widetilde{%
\Omega }_{i}$ for $i,\left( i=1,2,...\right) $, and that satisfies the
inequality: 
\begin{equation*}
\parallel g_{i}-\theta _{i}f_{A}\parallel _{W^{p,k}\left( \Omega ,%
%TCIMACRO{\U{211d} }%
%BeginExpansion
\mathbb{R}
%EndExpansion
\right) }\leq \frac{\beta }{2^{i+1}}
\end{equation*}%
for $i=0,1,2,...$.

Now let us consider the function $\psi :=\dsum\limits_{i=0}^{\infty }g_{i}$
and we claim that $\psi $ is a $C^{\infty }$-function over $\Omega $.

\ \ 

Indeed, for any open sub domain $\Delta \subset \subset \Omega $, we have $%
\psi _{m}:=\psi _{\shortmid \Delta }=\dsum\limits_{i=0}^{m}g_{i}$ for some $%
m\in 
%TCIMACRO{\U{2115} }%
%BeginExpansion
\mathbb{N}
%EndExpansion
$, since $\Delta \subset \subset \Omega $, we have that $\overline{\Delta }^{%
\text{cpct}}\varsubsetneq \Omega $ so that finitely many of the sets from
the cover $\left\{ V_{i}\right\} _{i}$ of $\Omega $ covers $\Delta $.
Therefore, for any set $\Delta \subset \subset \Omega $ and for a section $%
f_{A}$ of $f$ \ we have the inequality 
\begin{equation*}
\parallel \psi _{\mid \Delta }-\left( f_{A}\right) _{\mid \Delta }\parallel
_{\mid W^{p,k}\left( \Omega ,%
%TCIMACRO{\U{211d} }%
%BeginExpansion
\mathbb{R}
%EndExpansion
\right) }=\parallel \left( \dsum\limits_{i=0}^{\infty
}g_{i}-\dsum\limits_{i=0}^{\infty }\theta _{i}f_{A}\right) _{\mid \Delta
}\parallel _{\mid W^{p,k}\left( \Omega ,%
%TCIMACRO{\U{211d} }%
%BeginExpansion
\mathbb{R}
%EndExpansion
\right) }
\end{equation*}%
\begin{eqnarray*}
&=&\parallel \underset{\text{finite sum as }\Delta \subset \subset \Omega 
\text{ }}{\underbrace{\dsum\limits_{i=0}^{\infty }\left( g_{i}-\left( \theta
_{i}f_{A}\right) \right) }}\parallel _{\mid W^{p,k}\left( \Delta ,%
%TCIMACRO{\U{211d} }%
%BeginExpansion
\mathbb{R}
%EndExpansion
\right) } \\
&\leq &\dsum\limits_{i=0}^{m}\parallel g_{i}-\left( \theta _{i}f_{A}\right)
\parallel _{\mid W^{p,k}\left( \Omega ,%
%TCIMACRO{\U{211d} }%
%BeginExpansion
\mathbb{R}
%EndExpansion
\right) }\leq \dsum\limits_{i=0}^{\infty }\frac{\beta }{2^{i+1}}=\beta
\end{eqnarray*}%
where $f_{A}$ is represented by $f_{A}=\dsum\limits_{i=0}^{\infty }\theta
_{i}f_{A}$.

\ 

Therefore, considering the $\underset{\Delta \subset \Omega }{\text{sup}}%
\parallel \psi -f_{A}\parallel _{\mid W^{p,k}\left( \Delta ,%
%TCIMACRO{\U{211d} }%
%BeginExpansion
\mathbb{R}
%EndExpansion
\right) }$ we have the required result 
\begin{equation*}
\parallel \psi -f_{A}\parallel _{\mid W^{p,k}\left( \Omega \cup \partial
\Omega ,%
%TCIMACRO{\U{211d} }%
%BeginExpansion
\mathbb{R}
%EndExpansion
\right) }\leq \beta
\end{equation*}%
That is, the smooth function $\psi \left( =\underset{m\rightarrow \infty }{%
\lim }\psi _{m}=\underset{m\rightarrow \infty }{\lim }\left( \psi _{\mid
\Delta }\right) \right) $ approximates $f_{A}$ in the Sobolev space $%
W^{p,k}\left( \Omega \cup \partial \Omega ,%
%TCIMACRO{\U{211d} }%
%BeginExpansion
\mathbb{R}
%EndExpansion
\right) $.

\ \ 

Then since each $%
%TCIMACRO{\U{211d} }%
%BeginExpansion
\mathbb{R}
%EndExpansion
-$ valued component function $f_{A}$ of the $Cl_{n}-$valued function $%
f=\sum_{A}e_{A}f_{A}$ is smoothly approximated over $\overline{\Omega }%
(=\Omega \cup \partial \Omega )$ by $\psi _{A}\in C^{\infty }\left( \Omega
\cup \partial \Omega ,%
%TCIMACRO{\U{211d} }%
%BeginExpansion
\mathbb{R}
%EndExpansion
\right) $, we have that $\Psi =$ $\dsum\limits_{A}e_{A}\psi _{A}$
approximates the whole function $f$ over $\Omega \cup \partial \Omega $
which is $\overline{\Omega }.$ That is, we can make $\parallel f-\Psi
\parallel _{W^{p,k}\left( \Omega \cup \partial \Omega ,Cl_{n}\right) }$ as
small as we please.

\ \ 

Therefore, for $\epsilon >0$, and $A$ an index set, from the above argument,
we can make a component-wise $%
%TCIMACRO{\U{211d} }%
%BeginExpansion
\mathbb{R}
%EndExpansion
-$valued\ smooth approximation 
\begin{equation*}
\parallel f_{A}-\psi _{A}\parallel _{W^{p,k}}<\frac{\epsilon ^{p}}{2^{np}}
\end{equation*}%
on $\Omega \cup \partial \Omega $. The factor $2^{-n}$ in the last
inequality is related to the cardinality of a basis of the Clifford algebra $%
Cl_{n}$.

\ 

Then considering the functions $\Psi $ and $f$, with corresponding component
functions with the above corresponding sectional smooth approximations, we
have : 
\begin{equation*}
\parallel f-\Psi \parallel _{W^{p,k}\left( \Omega \cup \partial \Omega
,Cl_{n}\right) }=\parallel \dsum\limits_{A}e_{A}(f_{A}-\psi _{A})\parallel
_{W^{p,k}\left( \Omega \cup \partial \Omega ,Cl_{n}\right) }
\end{equation*}

\begin{equation*}
=\left( \dsum\limits_{A}(\parallel f_{A}-\psi _{A}\parallel _{W^{p,k}\left(
\Omega \cup \partial \Omega ,%
%TCIMACRO{\U{211d} }%
%BeginExpansion
\mathbb{R}
%EndExpansion
\right) }^{p})\right) ^{p^{-1}}<\epsilon
\end{equation*}
\end{proof}

\section{\textbf{Applications}}

In this section, we see the application of the two methods we discussed
above : approximation of a $\gamma -$regular function by minimal family of
functions and approximation of such a function by smooth functions.

\ \ 

The application of the complete and minimal function systems that we
constructed in approximating null solutions of first order partial
differential equations of the in-homogeneous Dirac operator is presented in
the following proposition.

\begin{proposition}
\cite{d1}Let $\Omega $ and $g$ be as in proposition $2$ . Then for a given $%
\varepsilon >0$ and for a given left $\gamma -$regular function $f$ given as
a solution of the BVP$\left( \ref{bvp}\right) $ in proposition $2$ ,\ there
exist Clifford numbers $\beta _{j}(j=1,...n_{0})$ such that 
\begin{equation*}
\Vert f-\sum_{j=1}^{n_{0}}\Psi _{j}^{\Gamma }\beta _{j}\Vert _{W_{\Gamma
,Cl_{n}}^{2,k}}<\varepsilon
\end{equation*}%
on $\Omega $.\ 
\end{proposition}

\begin{proof}
Since the system $\{\Psi _{m}^{\Gamma }\left( x\right) :=\frac{\overline{%
\left( x-x_{m}\right) }}{\omega _{n}\parallel x-x_{m}\parallel ^{n}}%
e^{-\Gamma \left( x-x_{m}\right) }\}_{m}$ is $Cl_{n}$-complete in the space
of left $\gamma $-regular functions which are in $W_{\Gamma }^{2,k}\left(
\Omega ,Cl_{n}\right) $ , where $\left\{ x_{m}\right\} _{m}$ is a dense
subset of some outer hypersurface $\Sigma _{out}$ of the domain $\Omega $\
such that $dist\left( \Sigma _{out},\partial \Omega \right) \geq \delta >0$
, the solution $f$ of the BVP$\left( \ref{bvp}\right) $ in proposition $2$,
can be approximated with finitely many elements of $\{\Psi _{m}^{\Gamma
}\}_{m}$ . That means, $\exists \beta _{j}\in Cl_{n}$ $(j=1,...,n_{0})$ such
that the above approximation inequality holds. The Clifford numbers $\beta
_{j}(j=1,...n_{0})$ are determined by solving a system of equations obtained
from the boundary conditions 
\begin{equation*}
tr_{\Sigma }\sum_{j=1}^{n_{0}}\Psi _{j}^{\Gamma }\beta _{j}(y_{i})=g(y_{i})\ 
\end{equation*}%
for each $i=1,...,n_{0}$, where $\{y_{i}:i=1,...,n_{0}\}$ is a set of
unisolvent points selected on $\Sigma $ as in proposition $9$.
\end{proof}

Then a best approximation of the above solution can be obtained from the
minimal functions.

\begin{corollary}
\cite{d1}Using the $Cl_{n}-$ minimal functions $\{\phi _{k}\}_{k}$, the
solution $f$ given by equation $\left( \ref{solu}\right) $\ of the \textbf{%
BVP} $\left( \ref{bvp}\right) $\ is approximated in the best way in $%
B_{(n_{0})}=\underset{Cl_{n}}{\text{span}}\left( \{\phi
_{j}\}_{j=1}^{n_{0}}\right) $ as 
\begin{equation*}
\parallel f-\sum_{j=1}^{n_{0}}\phi _{j}\lambda _{j}\parallel _{W_{\Gamma
}^{2,k}}<\varepsilon
\end{equation*}%
\ \ with $\lambda _{j}$ $\left( j=1,...,n_{0}\right) $ determined as in
proposition $11$.
\end{corollary}

\bigskip\ \ 

The next proposition gives the smooth approximation of a null solution of
the in homogeneous Dirac operator which is in a certain Sobolev space.

\begin{proposition}
Let $\Omega $ and $g$ be as in proposition $2$ . Then for a given $%
\varepsilon >0$ and for a given left $\gamma -$regular function $f$ given in 
$\left( \ref{solu}\right) $ as a solution of the BVP $\left( \ref{bvp}%
\right) $\ in proposition $2$, there exists a $C^{\infty }-$ function $\Psi
=\dsum\limits_{A}e_{A}\psi _{A}$ over $\Omega \cup \partial \Omega $ such
that 
\begin{equation*}
\parallel f-\Psi \parallel _{\mid W_{\Gamma }^{2,k}\left( \Omega \cup
\partial \Omega ,Cl_{n}\right) }<\varepsilon .
\end{equation*}
\end{proposition}

\begin{proof}
The analytic solution of the BVP$\left( \ref{bvp}\right) $ is given by a
boundary integral $\left( \ref{solu}\right) $ and this boundary integral
which is also written as $f=F_{\partial \Omega }\left( g\right) =F_{\partial
\Omega }\left( tr_{\partial \Omega }f\right) $ puts the solution in to the
Sobolev space $W^{2,k}\left( \Omega ,Cl_{n}\right) $.

\ \ 

This is because the trace operator as a sharpening operator ( that reduces
smoothness in this case by a $\frac{1}{2}$) has the property : 
\begin{equation*}
tr_{\partial \Omega }:W^{2,k}\left( \Omega ,Cl_{n}\right) \rightarrow W^{2,k-%
\frac{1}{2}}\left( \partial \Omega ,Cl_{n}\right)
\end{equation*}%
and the $\partial -$ integral as a left inverse of the $tr_{\partial \Omega
}-$operator as a mapping where the argument is a $\gamma -$regular function,
is a smoothening operator with the property :

\begin{equation*}
F_{\partial \Omega }=\left( \Psi ^{\Gamma }\underset{\text{convolution}}{%
\underbrace{\ast }}\nu \left( \cdot \right) \right) _{\mid \partial \Omega
}:W^{2,s}\left( \partial \Omega ,Cl_{n}\right) \rightarrow W^{2,s+\frac{1}{2}%
}\left( \Omega ,Cl_{n}\right) 
\end{equation*}

where, $\nu $ is the unit normal vector function defined on the boundary of $%
\Omega $.

\ \ \ 

But in general, the two operators, $\partial -$integral and $tr_{\partial
\Omega }-$ are inverses of each other in terms of preserving regularity, not
as function transformations.

\ 

Therefore, the solution function $f$ which is $Cl_{n}-$valued can be written
as

$\ f=\dsum\limits_{A}e_{A}f_{A}$, with%
\begin{equation*}
f_{A}:=\left( \int_{\partial \Omega }\Psi _{\Gamma }(x-y)\nu (y)g(y)d\Sigma
_{y}\right) _{A}
\end{equation*}%
the $A-$component of $f$.

\ \ \ 

But then as above, there exists a corresponding smooth Clifford valued
function $\psi _{A}$ so that for $\epsilon >0$, we have 
\begin{equation*}
\parallel f_{A}-\psi _{A}\parallel _{W^{2,k}\left( \Omega \cup \partial
\Omega ,%
%TCIMACRO{\U{211d} }%
%BeginExpansion
\mathbb{R}
%EndExpansion
\right) }\leq \frac{\epsilon ^{2}}{2^{2n}}
\end{equation*}

Therefore, by taking $\Psi $ as the Clifford sum of the component functions $%
\psi _{A}$, we have the following inequality:

\begin{eqnarray*}
&\parallel &f-\Psi \parallel _{W^{2,k}\left( \Omega \cup \partial \Omega
,Cl_{n}\right) }=\parallel \dsum\limits_{A}e_{A}\left( f_{A}-\psi
_{A}\right) \parallel _{W^{2,k}\left( \Omega \cup \partial \Omega
,Cl_{n}\right) } \\
&\leq &\dsum\limits_{A}\frac{\epsilon }{2^{n}}<\epsilon
\end{eqnarray*}
\end{proof}

\bigskip\ 

Interestingly enough, the smooth approximation works to BVPs which have
non-vanishing Dirac derivatives over the domain, unlike the minimal family
approximation which we have only for $\gamma -$regular functions with a non
vanishing trace.

We therefore give this result in the following proposition.

\begin{proposition}
Let $\Omega $ be a bounded domain in $%
%TCIMACRO{\U{211d} }%
%BeginExpansion
\mathbb{R}
%EndExpansion
^{n}$ whose boundary is a $C^{2}-$ hyper surface and let $g\in
W^{2,k-1}\left( \Omega ,Cl_{n}\right) ,h\in W^{2,k-\frac{1}{2}}\left(
\partial \Omega ,Cl_{n}\right) ,$ then the:

\begin{equation*}
\mathbf{NHBVP}:\left\{ 
\begin{array}{c}
D_{\gamma }f=g\text{, on }\Omega \\ 
trf=h,\text{ on }\partial \Omega%
\end{array}%
\right.
\end{equation*}

has a solution $f$ which is in the Sobolev space $W^{2,k}\left( \Omega
,Cl_{n}\right) $ given by 
\begin{equation*}
f\left( x\right) =\dint\limits_{\partial \Omega }\Psi ^{\Gamma }\left(
x-y\right) v\left( y\right) h\left( y\right) d\partial \Omega
_{x}+\dint\limits_{\Omega }\Psi ^{\Gamma }\left( x-y\right) g\left( x\right)
d\Omega _{x}
\end{equation*}

and therefore there exists a sequence $\left\{ \varphi _{m}\right\}
_{m=1}^{\infty }\subset C^{\infty }\left( \Omega \cup \partial \Omega
,Cl_{n}\right) $ such that for $\epsilon >0,\exists n_{0}\in 
%TCIMACRO{\U{2115} }%
%BeginExpansion
\mathbb{N}
%EndExpansion
$ $\ni $ 
\begin{equation*}
\Vert \varphi _{k}-\left( \dint\limits_{\partial \Omega }\Psi ^{\Gamma
}\left( x-y\right) v\left( y\right) h\left( y\right) d\partial \Omega
_{x}+\dint\limits_{\Omega }\Psi ^{\Gamma }\left( x-y\right) g\left( x\right)
d\Omega _{x}\right) \Vert _{W^{2,k}\left( \Omega \cup \partial \Omega
,Cl_{n}\right) }<\epsilon
\end{equation*}%
for all $k\geq n_{0}$.
\end{proposition}

\begin{proof}
First, one can see that the $\Omega -$ integral has the mapping property:

\begin{equation*}
\left( \Psi ^{\Gamma }\underset{\text{convolution}}{\underbrace{\ast }}%
\left( \cdot \right) \right) _{\mid \Omega }:W^{2,k}\left( \Omega
,Cl_{n}\right) \rightarrow W^{2,k+1}\left( \Omega ,Cl_{n}\right) 
\end{equation*}%
which is a smoothness augmentation by a one unlike the $\partial -$ integral
which increases by a half. \ \ 

\ \ \ 

Next, let%
\begin{equation*}
f_{A}:=\left( \dint\limits_{\partial \Omega }\Psi ^{\Gamma }\left(
x-y\right) v\left( y\right) h\left( y\right) d\partial \Omega
_{x}+\dint\limits_{\Omega }\Psi ^{\Gamma }\left( x-y\right) g\left( x\right)
d\Omega _{x}\right) _{A}
\end{equation*}%
the $A-$ Clifford \ section of \ $f$.

\ \ 

Then $f_{A}:\Omega \rightarrow 
%TCIMACRO{\U{211d} }%
%BeginExpansion
\mathbb{R}
%EndExpansion
$ is in the Sobolev section $W^{2,k}\left( \Omega ,%
%TCIMACRO{\U{211d} }%
%BeginExpansion
\mathbb{R}
%EndExpansion
\right) $ and therefore, $\exists $ a sequence $\left\{ \varphi
_{A,j}\right\} _{j=1}^{\infty }\subseteq C^{\infty }\left( \Omega \cup
\partial \Omega ,%
%TCIMACRO{\U{211d} }%
%BeginExpansion
\mathbb{R}
%EndExpansion
\right) $ such that for $\epsilon >0$, $\exists n_{A}\in 
%TCIMACRO{\U{2115} }%
%BeginExpansion
\mathbb{N}
%EndExpansion
$ such that for $m_{A}\geq n_{A}$ , where $m_{A}\in 
%TCIMACRO{\U{2115} }%
%BeginExpansion
\mathbb{N}
%EndExpansion
$, we have 
\begin{eqnarray*}
\Vert \varphi _{A,m_{A}}-f_{A}\Vert _{W^{2,k}\left( \Omega \cup \partial
\Omega ,%
%TCIMACRO{\U{211d} }%
%BeginExpansion
\mathbb{R}
%EndExpansion
\right) } &<&\frac{\epsilon ^{2}}{2^{2n}} \\
&&
\end{eqnarray*}%
Then for $n_{0}:=\max \left\{ m_{A}:\text{ }A\text{ is an index set}\right\} 
$ and for $k\geq n_{0}$, \ taking the Clifford valued function given by $%
\varphi _{k}:=\dsum\limits_{A}e_{A}\varphi _{A,k}$ which is $C^{\infty }-$%
over $\Omega \cup \partial \Omega $, we have 
\begin{eqnarray*}
\Vert f-\varphi _{k}\Vert _{W^{2,k}\left( \Omega \cup \partial \Omega
,Cl_{n}\right) } &=&\Vert \dsum\limits_{A}e_{A}\left( f_{A}-\varphi
_{A,k}\right) \Vert _{W^{2,k}\left( \Omega \cup \partial \Omega
,Cl_{n}\right) } \\
&\leq &\dsum\limits_{A}\frac{\epsilon }{2^{n}}<\epsilon
\end{eqnarray*}

that proves the proposition.
\end{proof}

\bigskip\ \ \ \ \ \ 

The next results focus on how far away are solutions of NHBVPs stated in \
proposition $6$, from space of monogenic functions or $\gamma -$regular
functions defined over the domain $\Omega $, if the input functions $g$ and $%
h$ are $C^{\infty }-$ over the respective domains of definition. We first
put Alexander's inequality for our purpose.

\ \ \ 

\begin{proposition}
\bigskip (Alexander)Let $f$ be a Clifford valued $C^{\infty }$- function
defined over a compact domain $\Omega $ in $%
%TCIMACRO{\U{211d} }%
%BeginExpansion
\mathbb{R}
%EndExpansion
^{n+1}$.

Then 
\begin{equation*}
\underset{C\left( \Omega ,Cl_{n}\right) }{dist}(f,M\left( \Omega
,Cl_{n}\right) )\leq \beta \left( \mu \left( \Omega \right) ^{\left( \frac{1%
}{n+1}\right) }\right) \Vert Df\Vert _{\infty }
\end{equation*}

where, $\mu $ is the volume measure in $%
%TCIMACRO{\U{211d} }%
%BeginExpansion
\mathbb{R}
%EndExpansion
^{n+1}$ and $\Vert \cdot \Vert _{\infty }$ is the supremum norm and $M\left(
\Omega ,Cl_{n}\right) $ is the set of Clifford valued functions defined over 
$\Omega $ which are annihilated by the Dirac differential operator $D$.$^{{}}
$
\end{proposition}

\bigskip\ \ \ \ 

From the above result of Alexander, we get the following important
inequality on solutions of NHBVPs.

\begin{proposition}
Let $\Omega $ be a compact domain in $%
%TCIMACRO{\U{211d} }%
%BeginExpansion
\mathbb{R}
%EndExpansion
^{n+1}$ and $g$ be a $C^{\infty }-$function over $\Omega $ and $h$ also be $%
C^{\infty }-$over $\partial \Omega $. Then the solution to the NHBVP:

\begin{equation*}
\left\{ 
\begin{array}{c}
D_{\gamma }f=g\text{, on }\Omega \\ 
trf=h\text{, on }\partial \Omega%
\end{array}%
\right.
\end{equation*}

satisfies the inequality: 
\begin{equation*}
\underset{C\left( \Omega ,Cl_{n}\right) }{dist}(f,M_{\gamma }\left( \Omega
,Cl_{n}\right) )\leq \beta \left( \mu \left( \Omega \right) ^{\left( \frac{1%
}{n+1}\right) }\right) \Vert g\Vert _{\infty }
\end{equation*}
where, $M_{\gamma }\left( \Omega ,Cl_{n}\right) $ is the set of Clifford
valued functions defined over $\Omega $ which are annihilated by the Dirac
like Differential operator $D_{\gamma }$.
\end{proposition}

\begin{proof}
From Borel-Pompeiu \ relation, the solution to the NHBVP given above is
given by the following integral equation: 
\begin{equation*}
f=\dint\limits_{\partial \Omega }\Psi ^{\Gamma }\left( x-y\right) \nu
trfd\partial \Omega +\dint\limits_{\Omega }\Psi ^{\Gamma }\left( x-y\right)
D_{\gamma }fd\Omega
\end{equation*}%
Using the input functions given on the domain and on the boundary, we have
the solution function to be : 
\begin{equation*}
f=\dint\limits_{\partial \Omega }\Psi ^{\Gamma }\left( x-y\right) \nu
hd\partial \Omega +\dint\limits_{\Omega }\Psi ^{\Gamma }\left( x-y\right)
gd\Omega
\end{equation*}%
Then by the inequality of Alexander, we have :

\begin{equation*}
\underset{C\left( \Omega ,Cl_{n}\right) }{dist}\left( \dint\limits_{\partial
\Omega }\Psi ^{\Gamma }\left( x-y\right) \nu hd\partial \Omega
+\dint\limits_{\Omega }\Psi ^{\Gamma }\left( x-y\right) gd\Omega ,M\left(
\Omega ,Cl_{n}\right) \right) \leq \beta \left( \mu \left( \Omega \right)
^{\left( \frac{1}{n+1}\right) }\right) \Vert g\Vert _{\infty }
\end{equation*}
\end{proof}

\ \ \ \ \ 

\begin{remark}
From the above inequality, one can see that if the domain is of measure
zero, then the solution is always approximated by monogenic functions, as
the indicated distance of the solution function from the family of monogenic
functions defined over $\Omega $ is zero for such a set.
\end{remark}

\begin{remark}
Besides, if the input function $g$ has a zero supremum norm then we have
also similar results.

But in a softer note, we see here a very important relation between the
supremum norm of the input function $g$ and how far is the solution function
away from monogenic functions. The thicker the supremum norm of the input
function, the farther away is the solution of the NHBVP from the family of
monogenic functions.
\end{remark}

\

\end{document}